\documentclass[a4paper,oneside,12pt]{article}

\usepackage{amsmath,amsfonts,amscd,amssymb,amsthm}
\usepackage[english]{babel}
\usepackage[utf8]{inputenc}
 \usepackage[T1]{fontenc}
\usepackage{graphicx}
\usepackage{caption}
\usepackage{subcaption}

\usepackage{bbm}
\usepackage{calrsfs}
\usepackage{mathtools}
\usepackage{color}
\usepackage{MnSymbol}
\usepackage{nicefrac}

\newtheorem{theorem}{Theorem}

\newtheorem{fact}{Fact}

\newtheorem{lemma}[theorem]{Lemma}
\newtheorem{proposition}[theorem]{Proposition}

\newcommand{\be}[1]{\begin{equation}\label{#1}}
\newcommand{\ee}{\end{equation}}
\numberwithin{equation}{section}

\newcommand{\ba}[1]{\begin{align}\label{#1}}
\newcommand{\ea}{\end{align}}
\numberwithin{equation}{section}

\newcommand{\ben}{\begin{equation*}}
\newcommand{\een}{\end{equation*}}
\numberwithin{equation}{section}

\newcommand{\calA}{\mathcal{A}}
\newcommand{\calB}{\mathcal{B}}

\newcommand{\calE}{\mathcal{E}}

\newcommand{\calX}{\mathcal{X}}

\newcommand{\bbZ}{\mathbb{Z}}

\newcommand{\ep}{\varepsilon}

\newcommand{\rk}[1]{\bgroup\color{red}%
  \par\medskip\hrule\smallskip%
  \noindent\textbf{#1}%
  \par\smallskip\hrule\medskip\egroup}

\newcommand{\ol}{\overline}

\newcommand{\lr}[1]{\xleftrightarrow{#1}}
\newcommand{\nlr}[2][]{\overset{#1}{\underset{#2}{\longleftrightarrow}} \kern -17pt
  \times \kern +7 pt}
\newcommand{\slab}{\mathbb{S}_k}

\renewcommand{\P}[1]{\mathbf{P}\!_p\left [ #1 \right]}
\mathtoolsset{showonlyrefs}

\title{Absence of infinite cluster for critical Bernoulli percolation on slabs}
\author{H. Duminil-Copin, V. Sidoravicius, V. Tassion}
\date{\today}

\begin{document}
\maketitle

\begin{abstract}
We prove that for Bernoulli percolation on a graph
$\bbZ^2\times\{0,\dots,k\}$ ($k\ge 0$), there is no infinite cluster
 at criticality,  almost surely. The proof extends to finite range Bernoulli
percolation models on $\bbZ^2$ which are invariant under $\pi/2$-rotation and reflection. 
\end{abstract}

\section{Introduction}

Determining whether a phase transition is continuous or discontinuous
is one of the fundamental questions in statistical physics. Bernoulli
percolation has offered the mathematicians a setup to develop
techniques to prove either continuity or discontinuity of the phase
transition, which in the case of continuity corresponds to the absence
of an infinite cluster at criticality. Harris \cite{Har60} proved that
the nearest neighbor bond percolation model with parameter $1/2$ on
$\mathbb Z^2$ does not contain an infinite cluster almost surely.
Viewed together with Kesten's result that $p_c \leq 1/2$
\cite{kesten1980critical}, it provided the first proof of such type of
statement. Since the original proof of Harris, a few alternative
arguments have been found for planar graphs (See, for example, a short
argument by Y. Zhang \cite[p 311]{Gri99}). In the late eighties,
dynamic renormalization ideas were successfully applied to prove
continuity in octants and half spaces of $\mathbb Z^d$, $d\ge 3$,
\cite{BarGriNew91a,BarGriNew91}. The continuity was also proved for
$\mathbb Z^d$ with $d\ge 19$ using the lace expansion technique
\cite{HS94}, and for non-amenable Cayley graphs using mass-transport
arguments \cite{BLPS99}. Despite all these developments, a general
argument to prove the continuity of the phase transition for the
nearest neighbor Bernoulli percolation on arbitrary lattices is still
missing, and the fact that the Bernoulli percolation undergoes a
continuous phase transition on $\mathbb Z^3$ still represents one of
the major open questions in the field.

This article provides the proof of continuity for
Bernoulli percolation on a class of non-planar lattices, namely slabs. 
We wish to highlight that the lattices $\mathbb Z^d$ with  $d\ge 3$ do not belong to this
class of graphs. 
\medbreak Consider the graph $\slab$, called {\em slab} of
  width $k$, given by the vertex set $\mathbb Z^2 \times \{0,\ldots,k\}$ and
edges between nearest neighbors. In what follows, ${\bf P}\!_p$ denotes the
Bernoulli bond percolation measure with parameter $p$ on $\slab$ defined as follows: every edge of $\mathbb Z^2 \times \{0,\ldots,k\}$ is {\em open} with probability $p$ (if it is not open, it is said to be {\em closed}) independently of the other edges. Let
$p_c(k)$ be the critical parameter of Bernoulli percolation on
$\slab$. Let $B$ be a subset of $\bbZ^3$, the event
$\{0\stackrel{B}{\longleftrightarrow}\infty\}$ denotes the existence
of an infinite path of open edges in $B$ starting from 0.

\begin{theorem}\label{thm:main theorem}
  For any $k>0$, ${\bf
    P}\!_{p_c(k)}[0\stackrel{\slab}\longleftrightarrow \infty]=0$.
\end{theorem}

For {\em site} percolation on $\mathbb S_2$, an {\em ad hoc} argument was provided in 
\cite{DNS12}. Nevertheless, one of the major difficulty of the present theorem is absent 
of \cite{DNS12}, namely the fact that ``crossing paths do not necessarily intersect''. 
This additional phenomenon, which is one of the main reasons why higher dimensional 
critical percolation is so difficult to study, requires the introduction of a new argument, 
based on the multi-valued map principle (see Lemma~\ref{lem:gluing} below for further explanations).

\paragraph{Two generalizations} The same proof works equally well (with suitable modifications) 
for any graph of the form $\mathbb Z^2\times G$, where $G$ is finite. 
This includes $G=\{0,\dots,k\}^{d-2}$ for $d\ge 3$. 

Similarly, symmetric finite range percolation on $\bbZ^2$ can be
treated via the same techniques (once again, relevant modifications
must be done). Let us state the result in this setting. Let ${\bf
  p}\in[0,1]^{\bbZ^2}$ be a set of edge-weight parameters, and  $M>0$.
We consider functions $\bf p$'s that are $M$-supported (meaning ${\bf
  p}_{z}=0$ for $|z|\ge M$) and invariant under reflection and
$\pi/2$-rotation (meaning that for all $z$, $\mathbf p_{\mathrm i
  z}=\mathbf p_{\bar z} =\mathbf p_z$). Consider the graph with vertex
set $\bbZ^2$ and edges between any two vertices and the
percolation ${\bf P}\!_{\bf p}$ defined as follows: the edge $(x,y)$ is
open with probability ${\bf p}_{x-y}$, independently of the other
edges.

\begin{theorem}\label{thm:finite range}
Fix $M>0$.  The probability ${\bf P}\!_{\bf p}[0\longleftrightarrow
\infty]$ is continuous, when viewed as a function defined on the set
of $M$-supported and invariant $\bf p$'s.
\end{theorem}

\paragraph{From the slab to $\bbZ^3$?} The fact that
$\bbZ^2\times\{0,\dots,k\}^{d-2}$ is approximating $\bbZ^d$ when $k$
tends to infinity suggests that the non-percolation on slabs could
shed a new light on the problem of proving the absence of infinite
cluster (almost surely) for critical percolation on $\bbZ^d$.
Nevertheless, we wish to highlight that this is not immediate. Indeed,
while $p_c(k)$ is known to converge to $p_c(\bbZ^3)$
\cite{grimmett1990supercritical}, passing at the limit requires a new
ingredient. For instance, a uniform control (in $k$) on the explosion
of the infinite-cluster density for $p$ tending to the critical point
would be sufficient.
 \begin{proposition}\label{prop:explosion}
Let $f:[0,1]\rightarrow\mathbb R$ be a continuous function such that $f(0)=0$. 
If for any $k\ge0$ and any $p\in(0,1)$,
$${\bf P}\!_p {[0\stackrel{\slab}\longleftrightarrow \infty]}\le f(p-p_c(k)),$$
then ${\bf P}\!_{p_c(\bbZ^3)}[0\stackrel{\bbZ^3}\longleftrightarrow \infty]=0$.
\end{proposition}

It is natural to expect that proving the existence of $f$ is roughly of the same 
difficulty as attacking the problem directly on $\bbZ^3$. Nevertheless, it could 
be that a suitable renormalization argument enables one to prove the existence 
of $f$.

Let us finish by recalling that several models undergo discontinuous phase 
transitions in high dimension and continuous phase transition in two 
dimensions (one may think of the 3 and 4-state Potts models). For 
most of these models, a discontinuous phase transition is expected 
to occur already in a slab. Theorem~\ref{thm:main theorem} shows 
that this is not the case for Bernoulli percolation. 

\paragraph{What about other models?} While this work is focused on the continuity of the phase transition for short range models, it is
well known that the complete picture of phase transition for
Bernoulli percolation is more complex. For one-dimensional long-range
Bernoulli systems with power law decay,
the transition may be discontinuous.  Indeed,  when the probabilities of edges of length $r$ being open decay as
$1/r^2$,  the percolation density at criticality is strictly positive, see \cite{AN86}.

Also, one may consider more general percolation models with dependence.
On $\bbZ^2$, the continuity of the phase transition was recently
proven \cite{DST14} for dependent percolation models known as random-cluster models with cluster-weight $q\in[1,4]$ (the special case $q=1$ corresponds to Bernoulli percolation). The continuity of the phase transition for $q = 1$ and $2$ was previously established
by Harris \cite{Har60} and Onsager \cite{Ons44} respectively.
Furthermore,  \cite{LaaMesMir91} showed that the phase transition is discontinuous for $q$ large enough.

Let us conclude this introduction by mentioning that  
 the phase transition on $\bbZ^d$ is expected to be discontinuous for $q>4$ when $d=2$ (we refer to \cite{Dum13} for details on this prediction), and for $q>2$ when $d\ge3$.
The best results (for $q>1$) in this direction are mostly restricted to integer values of $q$, for which the model is related to the Potts model. On the one hand, the fact that the phase transition is continuous for $q=2$ (corresponding to the Ising model) is known for any $d\ge 3$ \cite{AizDumSid13}. On the other hand for any $q\ge 3$, the random-cluster model undergoes a discontinuous phase transition above some dimension $d_c(q)$ \cite{BisChaCra06}. The proof of this result is based on Reflection-Positivity for the Potts model. 

\paragraph{Notation.} 
For a subset  $E$ of $\bbZ^2$, let $\overline {E}$ be the set of sites
in $\slab$ whose two first coordinates are in $E$.
 A \emph{cluster in $\overline E$} is a connected
component of the graph given by all the vertices in $\ol E$  and the
open edges with two endpoints in $\ol E$.
 Let $n$ be a positive integer,
$B$ a subset of $\mathbb Z^2$, and $X,Y\subset B$. We define
\begin{align}
  &X \lr{B} Y=\{\text{there exists an open cluster in $\overline {B}$ connecting
    $\overline{X}$ to $\overline{Y}$}\},\\
  &X \lr{!B!} Y=\{\text{there exists a \emph{unique} open cluster in $\overline {B}$ connecting
    $\overline{X}$ to $\overline {Y}$}\}.
\end{align}
Further we use the following notations: $B_n=[-n,n]^2$ and $\partial
B_n=B_n\setminus B_{n-1}$. \medbreak
\section{Proof}

\paragraph{Outline of the proof.} We follow a well known approach: we assume that ${\bf
  P}\!_p[0\stackrel{\slab}\longleftrightarrow \infty]>0$, and using this, we
  construct a finite-size criterion which is sufficient for percolation
to occur. By continuity, this finite-size criterion is satisfied for percolation
with parameters sufficiently close to $p$. This immediately implies
that ${\bf P}\!_{p_c}[0\stackrel{\slab}\longleftrightarrow \infty]=0$.

\bigskip

\noindent The proof is divided in three steps:
\begin{itemize}
\item First, we prove that ${\bf P}\!_p[0\stackrel{\slab}\longleftrightarrow \infty]>0$ 
  implies the
  existence of a certain event with a large probability. This step is
  new, and in particular, we invoke a gluing lemma to estimate probability of connections 
  between open paths.
\item The second step is classical. It consists in applying a block argument  to deduce that percolation occurs for any $q$ 
sufficiently close to $p$. 
\item The last step provides the proof of the gluing lemma. 
This lemma provides an answer   to a difficulty
encountered when doing renormalization in 3-dimensions 
(e.g. in \cite{grimmett1990supercritical}) {\em in the case of slabs}. When trying to 
construct long open connections by connecting two open 
paths together, the conditioning on the first path creates 
negative information along the path. As a consequence, 
one may construct open paths coming at distance one of 
the existing path, but the last edge can potentially be already 
explored and closed. This difficulty is one of the major obstacles 
in using a renormalization scheme to prove that 
${\bf P}_{p_c}[0\stackrel{\bbZ^3}\longleftrightarrow \infty]=0$. 
In our case, the fact that slabs are quasi-planar enables us 
to overcome this difficulty.\end{itemize}
\noindent
{\em From now on in this section, we fix $p$ and $k$ and 
we assume that $${\bf P}_p\big[0\stackrel{\slab}\longleftrightarrow \infty\big]>0.$$ 
Since the ambient space is fixed, we will not refer to $\slab$ 
and will rather write $X\longleftrightarrow Y$ instead of 
$X\stackrel{\slab}\longleftrightarrow Y$.}

\subsection{The finite-size criterion} 

The infinite cluster in $\slab$  being unique almost
surely~\cite{AKN87,burton1989uniqueness}, one can
construct a sequence $(u_n)_{n\ge1}$ such that $u_n\le n/3$ and 
\begin{align}
  \lim_{n\rightarrow \infty}\P{B_{u_n}\lr{!B_n!}\partial B_n}=1.\label{eq:1}
\end{align}
For simplicity, we set $S_n=B_{u_n}$. For 
$0\leq\alpha\leq\beta\leq n$, we define the following  event:
\begin{equation*}
 \calE_n(\alpha,\beta)=\big\{S_n\lr{B_n}\{n\}\times[\alpha,\beta]\big\}.
\end{equation*} 

\begin{figure}[h]
\centering
\begin{subfigure}{.52\textwidth}
  \centering
  \includegraphics[width=1.00\linewidth]{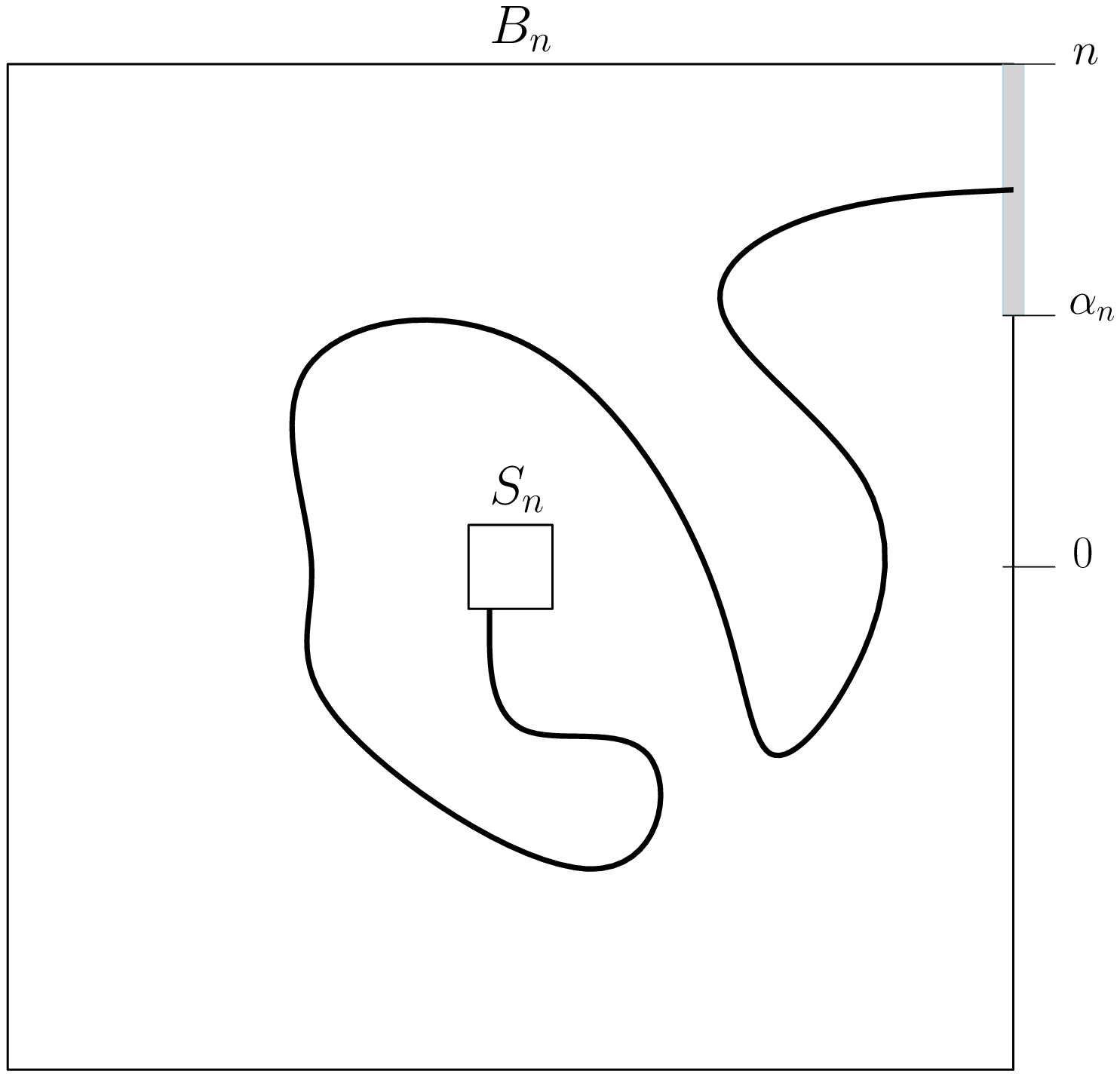}
  \caption{The event $\calE_n(\alpha_n,n)$.}
  \label{fig:sub1}
\end{subfigure}%
\begin{subfigure}{.52\textwidth}
  \centering
  \includegraphics[width=1.00\linewidth]{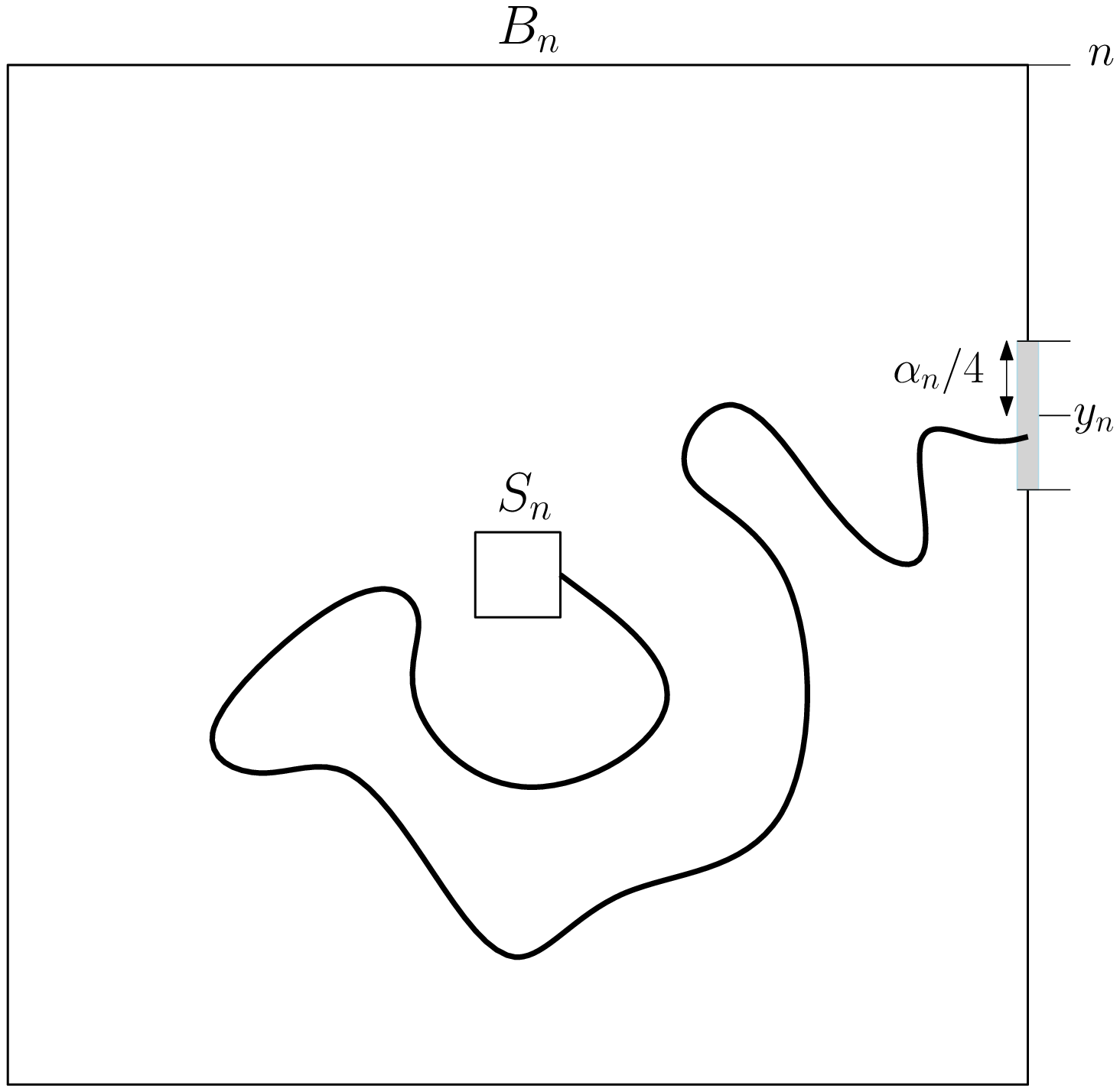}
  \caption{The event $\calE_n(y_n-\alpha_n/4,y_n+\alpha_n/4)$.}
  \label{fig:sub2}
\end{subfigure}
\medskip
\caption{The two events of Lemma~\ref{lem:crossing-with-high-proba}.}
\label{fig:test}
\end{figure}

\begin{lemma}
  \label{lem:crossing-with-high-proba}
  There exist two sequences $(y_n)$ and $(\alpha_n)$ with values in $[0,n]$, such that
  \begin{align}
    &\lim_{n\to\infty}\P{\calE_n(\alpha_n,n)}=1,\\
    &\lim_{n\to\infty}\P{\calE_n(y_n-\alpha_n/4,y_n+\alpha_n/4)}=1.
  \end{align}
\end{lemma}
The proof relies on the following classical inequality, which is a
straightforward consequence of the Harris-FKG inequality. Let
$\calA_1,\dots,\calA_m$ be $m$ increasing events. Then
\begin{equation}
\max_{i=1,\dots,m}\P{\calA_i}\ge
1-(1-\P{\calA_1\cup \dots\cup \calA_m})^{1/m}.\label{eq:3}
\end{equation}
When the events are of equal probability, this inequality is known as
``square-root trick''. We use the same name for the
generalization given by~\eqref{eq:3}.

\begin{proof}[Proof of Lemma~\ref{lem:crossing-with-high-proba}]

Applying the square-root trick and using the symmetries of the box, we obtain
\begin{equation*}
  \P{\calE_n(0,n)}\geq 1-\left(1-\P{S_n\lr{B_n}\partial B_n}\right)^{1/8}
  \end{equation*}
  which implies that $\P{\calE_n(0,n)}$ also tends to 1 as $n$ goes to infinity.
  Now,  for $\alpha \in\{0,\ldots, n-1\}$ we will use  the decomposition  
  \begin{equation}
        \calE_n(0,n)=\calE_n(0,\alpha)\cup\calE_n(\alpha+1,n).
  \end{equation}
  The probability of the event $\calE_n(0,0)$ is smaller than  some constant $c<1$ uniformly in $n$ and $\P{\calE_n(0,n)}$ tends to $1$, providing
  that for $n$ large enough:
  \begin{equation}
         \P{\calE_n(0,0)}<\P{\calE_n(1,n)}.
  \end{equation}
  In the same way, we also have for $n$ large enough
   \begin{equation}
         \P{\calE_n(0,n-1)}>\P{\calE_n(n,n)}.
  \end{equation}
  The two inequalities above ensure that the inequality between
  $\P{\calE_n(0,\alpha-1)}$ and $\P{\calE_n(\alpha,n)}$ reverses for a
    non-trivial $\alpha$. More precisely we can define
    $\alpha_n\in\{1,\ldots,n-1\}$ by
  \begin{equation*}
    \alpha_n=\max\big\{\alpha\le n-1 : \:
    \P{\calE_n(0,\alpha-1)}<\P{\calE_n(\alpha,n)}\big\},
  \end{equation*}
  and this choice implies that
  \begin{equation}
    \P{\calE_n(0,\alpha_n-1)}<\P{\calE_n(\alpha_n,n)}\text{ and
    }\P{\calE_n(0,\alpha_n)}\ge\P{\calE_n(\alpha_n+1,n)}.
  \end{equation}
Therefore, two other uses of the square-root trick imply that
$\P{\calE_n(0,\alpha_n)}$ and $\P{\calE_n(\alpha_n,n)}$ are larger
than $1-(1-\P{\calE_n(0,n)})^{1/2}$ and thus tends to $1$ when $n$
goes infinity.
Finally, we decompose
\begin{equation}
  \calE_n(0,\alpha_n)= \calE_n(0,\alpha_n/2)\cup\calE_n(\alpha_n/2,\alpha_n)
\end{equation}
and a last application of the square root trick allows to define
$y_n=\alpha_n/4$ or $y_n=3\alpha_n/4$ such that
\begin{equation}
\P{\calE_n(y_n-\alpha_n/4,y_n+\alpha_n/4)}\ge 1- \sqrt{1-\P{ \calE_n(0,\alpha_n)}},
\end{equation}
which concludes the proof of the lemma. 
\end{proof}

\begin{lemma}\label{lem:infinite number}
  There exist infinitely many $n$ such that $\alpha_{3n}\leq 4\alpha_n$.
\end{lemma}
\begin{proof}
A sequence of positive integers such that $\alpha_{3n}>4\alpha_n$ for $n$ 
large enough grows super-linearly. Since $\alpha_n\le n$, we obtain the 
result.\end{proof}
\begin{figure}[h]
  \centering
  \includegraphics[width=9cm]{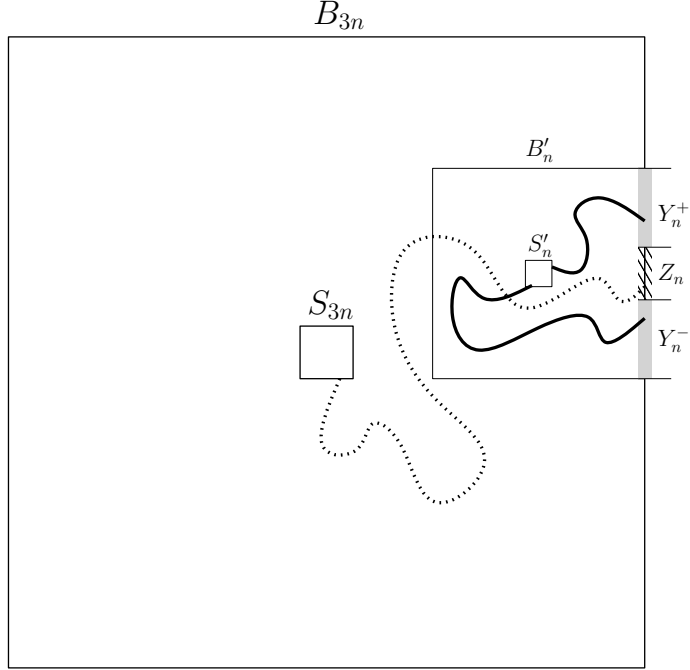}
  \caption{The events $S_{3n}\lr{B_{3n}}Z_n$ (the path is depicted by dots) 
  and $\{S_n'\lr{B_n'}Y_n^-\}\cap\{S_n'\lr{B_n'}Y_n^+\}$ (the paths are 
  depicted in bold).}
  \label{fig:AB}
\end{figure}
Let $n\ge 1$. Write $y=y_{3n}$ and define the following five subsets
of $\bbZ^2$ (see Fig.~\ref{fig:AB}  
for an illustration):
\begin{align*}
  B_n'&=(2n,y)+B_n,\\
 S_n'&=(2n,y)+S_n,\\
  Y_n^+&=\{3n\}\times[y+\alpha_n,y+n],\\
  Y_n^-&=\{3n\}\times[y-n,y-\alpha_n],\\
  Z_n&=\{3n\}\times[y-\alpha_n,y+\alpha_n].
\end{align*}
 When $n$ is  such that $\alpha_{3n}/4\le\alpha_n$, we have
 \begin{equation}
 \P{S_{3n}\lr{B_{3n}}Z_n} \ge \P{\calE_{3n}(y_{3n}-\alpha_{3n}, y_{3n}+\alpha_{3n})},
\end{equation}
and Lemmata~\ref{lem:crossing-with-high-proba} and \ref{lem:infinite number} imply that
\begin{equation}
  \label{eq:2}
    \limsup_{n\to \infty}\P{S_{3n}\lr{B_{3n}}Z_n}=1.
\end{equation}

Using Harris inequality and the invariance of $\mathbf P_p$ under reflection, we deduce that 
\begin{align*}
\P{S_{3n}\lr{B_{3n}}Z_n,S_n'\lr{B_n'}Y_n^-,S_n'\lr{B_n'}Y_n^+}&\ge \P{S_{3n}\lr{B_{3n}}Z_n}\P{\calE_n(0,\alpha_n)}^2.
\end{align*}
From Lemma~\ref{lem:crossing-with-high-proba} and
\eqref{eq:2}, we finally obtain
\begin{equation}\label{eq:crucial 1}
\limsup_{n\rightarrow \infty} \P{S_{3n}\lr{B_{3n}}Z_n,S_n'\lr{B_n'}Y_n^-,S_n'\lr{B_n'}Y_n^+}=1.
\end{equation}
 
We now intend to construct a path from $\ol{S_{3n}}$ to $\ol{S'_n}$. Projections of
paths from $\ol{S_{3n}}$ to $\ol{Z_n}$ and from $\ol{S_n'}$ to $\ol{Y_n^-}$ and $\ol{Y_n^+}$
must intersect (as illustrated on Fig.~\ref{fig:AB}), but the paths themselves 
have no reason to do so. This is one of the main difficulties when working 
with non-planar graphs.  Let us assume for a moment that we have the following 
lemma at our disposition and let us finish the proof. Note that this lemma is 
a crucial ingredient of the proof, since it solves the problem of the intersection 
of paths on slabs. 

\begin{lemma}[Gluing Lemma]\label{lem:gluing}
  For any $\ep>0$, there exists $\delta=\delta(\ep,k)>0$ such that for any $n$, 
$$\P{S_{3n}\lr{B_{3n}}Z_n,S_n'\lr{B_n'}Y_n^-,S_n'\lr{B_n'}Y_n^+} \ge 1-\delta$$
implies
$$ \P{S_{3n}\lr{B_{3n}\cup B_n'}S_n'}\ge 1-\ep.$$
\end{lemma}

\vspace{.5cm}

\noindent Lemma~\ref{lem:gluing} and \eqref{eq:crucial 1} imply that
\begin{equation}\label{eq:12}
\limsup_{n\rightarrow \infty}\P{S_{3n}\lr{B_{4n}}S'_n}=1.
\end{equation}
Observe that
\begin{align*}
&\P{S_{3n}\lr{(2n,0)+B_{6n}}(4n,0)+S_{3n}}\\
&\ge
\P{\{S_{3n}\lr{B_{4n}}S'_n\}\cap\{S'_n\lr{(4n,0)+B_{4n}}(4n,0)+S_{3n}\}\cap\{S'_{n}\lr{!B_{n}'!}\partial
  B_{n}'\}}\\
&\ge
\P{\{S_{3n}\lr{B_{4n}}S'_n\}\cap\{S'_n\lr{(4n,0)+B_{4n}}(4n,0)+S_{3n}\}}+\P{S'_{n}\lr{!B_{n}'!}\partial
  B_{n}'}-1\\
&\ge \P{S_{3n}\lr{B_{4n}}S'_n}^2+\P{S'_{n}\lr{!B_{n}'!}\partial B_{n}'}-1.
\end{align*}
The first inequality followed from the fact that paths coming from
$\ol{S_{3n}}$ and $\ol{(4n,0)+S_{3n}}$ and going to $\ol{S'_n}$ must
be connected to each other in $\ol{B'_n}$ by uniqueness of the cluster
in $\overline{B_n'}$ from $\ol{S'_n}$ to $\ol{\partial B'_n}$. The
Harris inequality and the reflection across the axis $\{2n\}\times
\mathbb R$ were used in the last inequality.

Using \eqref{eq:12} and \eqref{eq:1}, we find\begin{align}\label{eq:almost good}
\limsup_{n\rightarrow \infty}\P{S_{3n}\lr{(2n,0)+B_{6n}}(4n,0)+S_{3n}}=1.\end{align}

\subsection{The renormalization step} 

Fix $n\in \mathbb N$ to be chosen below. Call an edge $\{z,z'\}$ of $4n\bbZ^2$ {\em good} if 
\begin{itemize}
\item $z+S_{3n}\lr{R_n}z'+S_{3n}$, with $R_n=\frac{z+z'}2+B_{6n}$,
\item $z+S_{3n}\lr{!z+B_{3n}!}z+\partial B_{3n}$ and
  $z'+S_{3n}\lr{!z'+B_{3n}!}z'+\partial B_{3n}$.
\end{itemize}
Notice that the set of good edges follows a percolation law which is \linebreak
$4$-dependent. In particular, there exists $\eta>0$ such that whenever
the probability to be good exceeds $1-\eta$, the set of good edges
percolates (this fact follows from a Peierls argument presented for
example in \cite[Lemma 1]{balister2005continuum}, or from
the classical result of~\cite{ligget1997domination} comparing $4$-dependent 
percolation to Bernoulli percolation).

Equations \eqref{eq:almost good} and \eqref{eq:1} guarantee the
existence of $n$ such that the ${\bf P}_p$-probability that an edge is
good is larger than $1-\eta$. Since being good depends only on the
state of the edges in a finite box, there exists $q<p$ such that an
edge is good with ${\bf P}_q$-probability larger than $1-\eta$, and
the set of good edges percolates for the percolation of parameter $q$.

By construction, an infinite path of good edges in the coarse-grained 
lattice immediately implies the existence of an infinite path of open 
edges in the original lattice. As a consequence, $q\ge p_c(k)$ and 
therefore $p>p_c(k)$. This concludes the proof of 
${\bf P}_{p_c(k)}[0\longleftrightarrow \infty]=0$ conditionally on 
Lemma~\ref{lem:gluing}.

\subsection{The proof of Lemma~\ref{lem:gluing} (the gluing Lemma)}

First, observe that the lemma holds trivially for $k=0$ by setting 
$\delta(\ep,0)=\ep$. {\em We therefore assume from now on that $k\ge 1$}. 
We will be using the following lemma.
\begin{lemma}\label{lem:multi-valued map}
 Let $s,t>0$. Consider two events $\calA$ and $\calB$ and a map $\Phi$
  from $\calA$ into the set $\mathfrak{P}(\calB)$ of subevents of
  $\calB$. We assume that:
  \begin{enumerate}
  \item for all $\omega\in \calA$, $|\Phi(\omega)|\geq t$,
  \item for all $\omega' \in \calB$, there exists a set $S$ with less than 
  $s$ edges such that $\{\omega: \omega' \in \Phi(\omega)
    \} \subset \{\omega : \omega_{|_{S^c}}=\omega'_{|_{S^c}}\}$.
  \end{enumerate}
  Then,
  \begin{equation*}
    \label{eq:7}
    \P{\calA}\leq\frac{(2/\min\{p,1-p\})^s}{t}\P{\calB}.
  \end{equation*}
\end{lemma}
This lemma will enable us to bound from above the probability of $\calA$ when $s$
is small and $t$ is large.
\begin{proof}
  It follows from exchanging the order of the summation on $\omega$
  and on $\omega'\in \Phi(\omega)$:
  \begin{align}
    \sum_{\omega\in\calA}\P{\omega}&\le\frac1{t(\min\{p,1-p\})^s}
    \sum_{\omega\in \calA}\P{\Phi(\omega)}\\
    &=\frac1{t(\min\{p,1-p\})^s}\sum_{\omega'\in
      \calB}\mathrm{Card}\{\omega:\omega'\in
    \Phi(\omega)\}\cdot\P{\omega'}\\
    &\le \frac{2^s}{t(\min\{p,1-p\})^s}\sum_{\omega'\in
      \calB}\P{\omega'}.
  \end{align}
\end{proof}

Let us now explain how the previous statement can be used to prove Lemma~\ref{lem:gluing}.
Fix an arbitrary order $\prec$ on edges emanating from each vertex
of $\slab$, which is invariant under translations of $\bbZ^2$. Also fix an arbitrary order $\ll$ on vertices of $\slab$. Then,
define a total order on self-avoiding paths from $\ol{S_{3n}}$ to
$\ol{Z_n}$ by taking the lexicographical order: for two paths
$\gamma=(\gamma_i)_{i\le r}$ and $\gamma'=(\gamma'_i)_{i\le r'}$, we
set $\gamma < \gamma'$ if one of the following conditions occurs:
\begin{itemize}
\item $r< r'$ and $\gamma=(\gamma'_i)_{i\le r}$,
\item $\gamma_0\ll \gamma_0'$,
\item there exists $k<\min\{r,r'\}$ such that $\gamma_j=\gamma'_j$ for 
  $j\le k$ and $(\gamma_k,\gamma_{k+1})\prec
  (\gamma'_k,\gamma'_{k+1})$.
\end{itemize}

\paragraph{Definition.}{\em Consider $\omega$ with at least one open path from 
$\ol{S_{3n}}$ to $\ol{Z_n}$. Define $\gamma_\mathrm{min}(\omega)$ to 
be the minimal (for the order defined above) open self-avoiding path
  from $\ol{S_{3n}}$ to $\ol{Z_n}$.  
  Let $U(\omega)$ be the set of points $z$ in $B'_n$ with 
  \begin{itemize}
  \item[{\bf P1}] $\ol{\{z\}} \cap \gamma_\mathrm{min}(\omega)\ne \emptyset$,
  \item[{\bf P2}] $\ol{z+B_1}$ is connected to $\ol{S'_n}$ by an open
    path $\pi$, such that the distance between the canonical
    projections of $\pi$ and $\gamma_{\rm min}$  onto
    $\bbZ^2$ is exactly $1$.
  \end{itemize}}

\noindent  Write
  $\calX=\{S_{3n}\lr{B_{3n}}Z_n,S_n'\lr{B_n'}Y_n^-,S_n'\lr{B_n'}Y_n^+\}
  \cap \{S_{3n} \lr{B_{3n}\cup B_n'} S_n'\}^c$. Proving Lemma~\ref{lem:gluing} corresponds to
  proving that the probability of $\calX$ is small whenever the
  probability of $\{S_{3n}\lr{B_{3n}}Z_n,S_n'\lr{B_n'}Y_n^-,S_n'\lr{B_n'}Y_n^+\}$ is close to 1. We proceed in two steps, depending on whether the cardinality of $U(\omega)$ is large or not.
    \begin{fact}
    \label{fact:1}
    Fix $\ep>0$ and $t>0$. There exists $\delta>0$ so that $$\P{S_n'\lr{B_n'}Y_n^-,S_n'\lr{B_n'}Y_n^+}>1-\delta$$ implies
      $\P{\calX\cap \{|U|< t\}}\leq \varepsilon$.
  \end{fact}
  \begin{proof}[Proof of Fact~\ref{fact:1}]
    Let $\omega \in \calX$ such that $|U(\omega)|<t$. Define $\omega'$
    to be the configuration obtained from $\omega$ by closing, for any
    $z\in U(\omega)$, all the edges $\{u,v\}$ such that
    $u\in\ol{\{z\}}$ and $v$ is connected to $\ol{S_n'}$ by an open
    path.
    
    Observe that $\omega'$ cannot contain two open paths in
    $\ol{B'_n}$ from $\ol{S'_n}$ to $\ol{Y_n^-}$ and $\ol{Y_n^+}$
    respectively. Indeed, an open path in $\omega'$ must be in
    $\omega$. Furthermore, two paths from $\ol{S'_n}$ to $\ol{Y_n^-}$
    and $\ol{Y_n^+}$ respectively must intersect at least one set of
    the form $\ol{\{z\}}$ with $z$ in $U(\omega)$. But this implies
    that one edge of one of these two paths was turned to closed in
    $\omega',$ which is a contradiction. We therefore constructed a map

    \begin{equation*}
      \label{eq:8}
      \begin{array}{lccc}
        \Phi:&\calX\cap \{|U|<
        t\}&\longrightarrow&\{S_n'\lr{B_n'}Y_n^-,S_n'\lr{B_n'}Y_n^+\}^c
      \end{array}
    \end{equation*}
    mapping a configuration $\omega$ to $\omega'$. For any $\omega'$
    in the image of $\Phi$, the set $\{\omega:\Phi(\omega)=\omega'\}$
    contains only configurations that are equal to $\omega'$ except
    possibly on the edges adjacent to $U (\omega')$. Here, we use the
    fact that $U(\omega')=U(\omega)$ and
    $\gamma_{\mathrm{min}}(\omega')=\gamma_{\mathrm{min}}(\omega)$ for
    any pre-image of $\omega'$ (since {\bf P1} guarantees that no edge
    of $\gamma_{\mathrm{min}}(\omega')$ was closed in the process).
    Lemma~\ref{lem:multi-valued map} can be applied to obtain
    \begin{equation*}
      \label{eq:10}
      \P{\calX\cap \{|U|< t\}}\leq (2/\min\{p,1-p\})^{6kt} \P{\left\{S_n'\lr{B_n'}Y_n^-,S_n'\lr{B_n'}Y_n^+\right\}^c}.
    \end{equation*}
    Fact~\ref{fact:1} follows immediately. \end{proof}

    \begin{fact}
    \label{fact:2}
     Fix $\ep>0$. For $t$ large enough,
     \begin{equation}
    \P{\calX\cap \{|U| \ge t\}}\leq  \ep \P{S_{3n} \lr{B_{3n}\cup B_n'} S_n'}.\label{eq:11}
    \end{equation}
  \end{fact}
  \begin{proof}[Proof of Fact~\ref{fact:2}]
    For $R\ge 1$ and $z=(z_1,z_2,z_3) \in \slab$, we write
    $\ol{B_R}(z)$ for $\ol{(z_1,z_2)+B_R}$. Fix $R\ge2$ in such a way
    that for any site $z\in \slab$, for any three distinct neighbors
    $u$, $v$, $w$ of $z$ and any three distinct sites $u'$, $v'$, $w'$
    on the boundary of $\ol{B_R}(z)$, there exist three disjoint
    self-avoiding paths in $\ol{B_R}(z)\setminus\{z\}$ connecting $u$
    to $u'$, $v$ to $v'$ and $w$ to $w'$. Note that such an $R$ exists
    since in this section, $k$ is assume to be strictly larger than 0.
    \medbreak
    \noindent {\bf Remark. }{\em For the slab, one could take $R=2$.
      Nevertheless, taking larger $R$ becomes necessary when dealing
      with finite range percolation. Since the proof is not more
      complicated, we choose to present it with an arbitrary $R$.}
  \medbreak
  \begin{figure}[h]
  \centering
  \includegraphics[width=7cm]{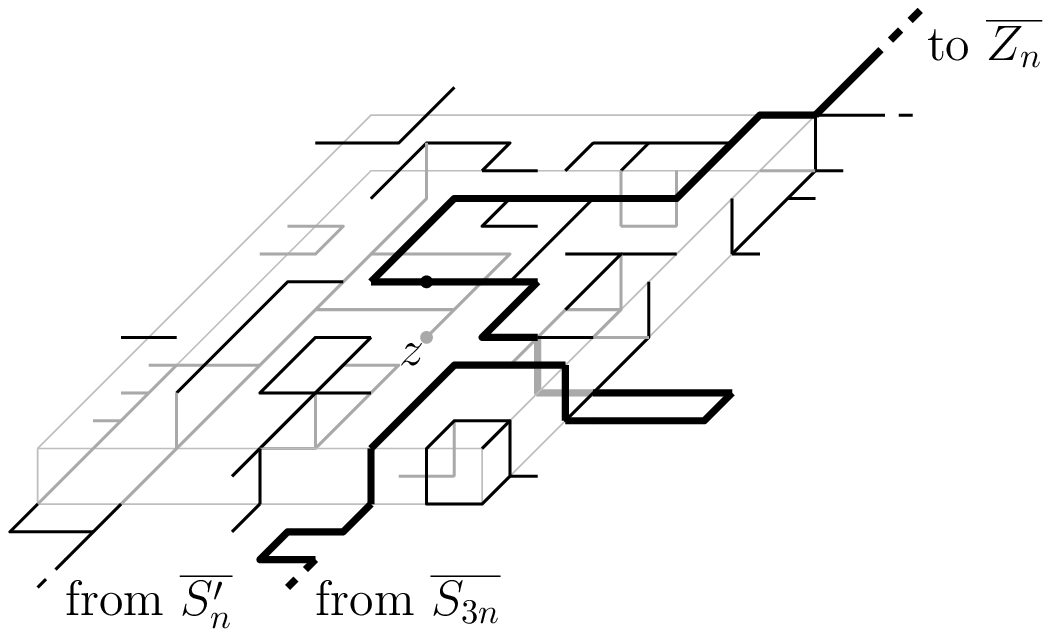}
  \includegraphics[width=7cm]{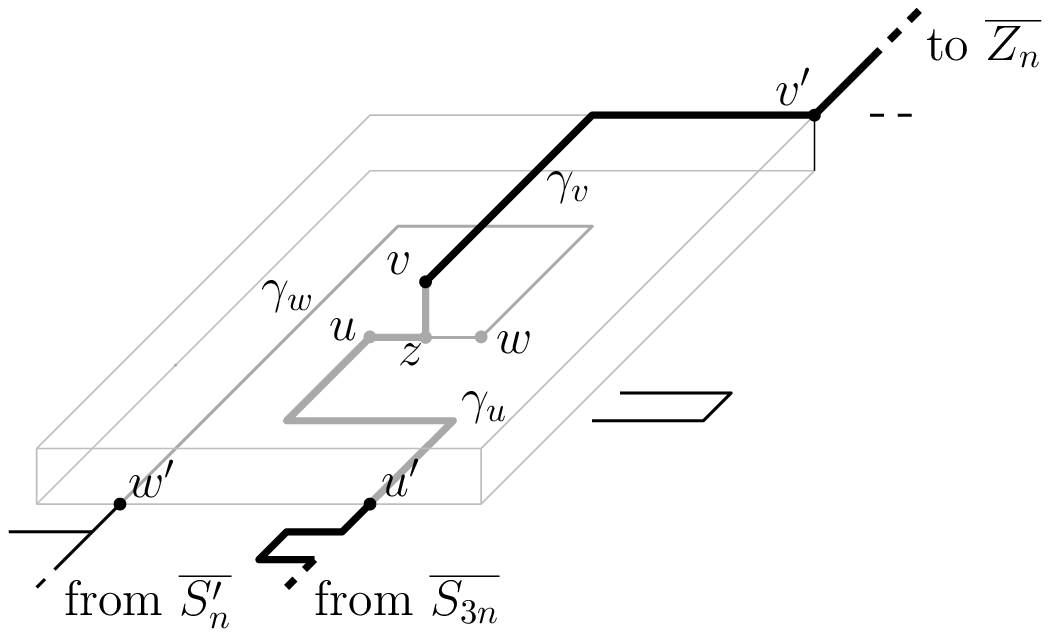}
  \caption{Two configurations $\omega$ and $\omega^{(z)}$. In both
    cases, $\gamma_{\mathrm{min}}$ is depicted in bold, and closed
    edges are not drawn for clarity. Note that at the end of the
    construction, there are exactly three open edges connecting a
    vertex of $\ol{B_R}(z)$ to a vertex in the complement of
    $\ol{B_R}(z)$.}
  \label{fig:last}
\end{figure}
    Fix $\omega\in \calX$ such that  $|U(\omega)| \ge t$ and pick $z\in U(\omega)$.  Construct the configuration $\omega^{(z)}$ as follows (see Fig.~\ref{fig:last} for an illustration of the construction):
   \begin{enumerate}
    \item Choose $u,v,w$ in such a way that $(z,u)$, $(z,v)$ and $(z,w)$ are three distinct edges with $(z,v)\prec(z,w)$. 
    
    Define $u'$ and $v'$ to be respectively the first and last (when going from $S_{3n}$ to $Y_n^+\cup Y_n^-$) vertices of $\gamma_{\mathrm{min}}(\omega)$ which are in $\ol{B_R}(z)$ (these two vertices exist and are distinct since $\gamma_{\mathrm{min}}(\omega)$ intersects the set $\ol{B_1}(z)$ by {\bf P2}). 
    
    Choose $w'$ on the boundary of $\ol{B_R}$ in such a way that
    there exists an open self-avoiding path $\pi$ from $w'$ to
    $\ol{S_n'}$, all the edges of which lie outside $\ol{B_R}(z)$ (this path exists by {\bf P1}). Since $\omega\in\calX$, we also have that $w'$ is different from $u'$ and $v'$ (otherwise $S_{3n}\longleftrightarrow S'_n$ in $\omega$).
      
\item Close all edges of $\omega$ in $\ol{B_{R+1}}(z)$ at the exception of the edges of $\ol{B_{R+1}}(z)\setminus\ol{B_R}(z)$ which are in $\gamma_{\mathrm{min}}(\omega)$ or $\pi$.

\item Open the edges $(z,u)$, $(z,v)$ and $(z,w)$, together with three disjoint self-avoiding paths $\gamma_u$, $\gamma_v$ and $\gamma_w$ in $\ol{B_R}(z)\setminus\{z\}$ connecting $u$ to $u'$, $v$ to $v'$, and $w$ to $w'$. 
\end{enumerate}
By construction, $\omega^{(z)}$ is in $\{S_{3n} \lr{B_{3n} \cup B'_n} S'_n\}$ and we can define the map
    \begin{equation*}
      \label{eq:9}
       \begin{array}{lccc}
        \Psi:&\calX\cap \{|U|> t\}&\longrightarrow&\mathfrak{P}(S_{3n} \lr{B_{3n} \cup B'_n} S'_n)\\
        &\omega&\longmapsto&\{\omega^{(z)},\,z\in U(\omega) \}.
      \end{array}
    \end{equation*}
We wish to apply Lemma~\ref{lem:multi-valued map}. In order to do so, the following observation will be useful. 

Working with the lexicographical order implies that $\gamma_{\mathrm
  {min}}(\omega^{(z)})$ and $\gamma_{\mathrm {min}}(\omega)$
necessarily coincide up to $u'$. Thanks to the second step, the degree
of $u'$ in $\omega^{(z)}$ is 2. This fact forces any self-avoiding
open path from $\ol{S_{3n}}$ to $\ol{Z_n}$ containing the minimal path
up to $u'$ to contain $\gamma_u$. Now (this is the crucial point of
the construction), we have that $(z,v)\prec(z,w)$. Therefore, even
though there could exist an open path from $z$ to $\ol{Z_n}$ passing
by $w$, {\em the minimal path will still be going through $v$}. Hence,
the continuation of the minimal path goes through $v$ and thus
contains $\gamma_v$ for the same reason that it was including
$\gamma_u$. From $v'$, the minimality of $\gamma_{\mathrm
  {min}}(\omega)$ implies that $\gamma_{\mathrm {min}}(\omega^{(z)})$
and $\gamma_{\mathrm {min}}(\omega)$ coincide from this vertex up to
the end.

Since no site of $\gamma_{\mathrm {min}}(\omega)$ is connected to
$\ol{S'_n}$ in $\omega$ (simply because $\omega\in \calX$), the previous
paragraph implies that $z$ is {\em the only site on $\gamma_{\mathrm
    {min}}(\omega^{(z)})$ to be connected to $\ol{S'_n}$ without using any
  edge in $\gamma_{\mathrm {min}}(\omega^{(z)})$}. \medbreak We are
now in a position to apply Lemma~\ref{lem:multi-valued map}. The
configurations $\omega^{(z)}$ are all distinct since either
$\gamma_{\mathrm {min}}(\omega^{(z)})\ne \gamma_{\mathrm
  {min}}(\omega^{(z')})$ (which readily implies that the configurations
are distinct), or $\gamma_{\mathrm
  {min}}(\omega^{(z)})=\gamma_{\mathrm {min}}(\omega^{(z')})$ but then
$z=z'$ by the characterization of $z$ (and $z'$) above.

Furthermore, consider a pre-image $\omega$ of $\omega'$ and assume
that $\omega'=\omega^{(z)}$ for some $z\in \slab$. The discussion
above shows that $z$ is determined uniquely. Beside, the
configurations  $\omega$ and
$\omega^{(z)}$ differ only in $\ol{B_{R+1}}(z)$.

In conclusion, the map $\Phi$ verifies the hypotheses of Lemma~\ref{lem:multi-valued map} with
    $s$ equal to the number of edges in $\ol{B_{R+1}}$. This gives
    \begin{equation*}
      \label{eq:6}
      \P{\calX\cap \{|U| > t\}}\leq \frac{(2/\min\{p,1-p\})^C}{t} \P{S_{3n} \lr{B_{3n}\cup
          B'_n} S'_n}.
    \end{equation*}
    Choosing $t$ large enough concludes the proof.
 \end{proof}
\noindent Fix $\ep>0$. Choosing first $t$ as in Fact~\ref{fact:2} and
then $\delta$ as in Fact~\ref{fact:1} conclude the proof of Lemma~\ref{lem:gluing}.

\subsection{The proof of Proposition~\ref{prop:explosion}}

\begin{proof}[Proof of Proposition~\ref{prop:explosion}]Recall the result of \cite{grimmett1990supercritical} yielding that $p_c(k)$ tends to $p_c(\bbZ^3)$ as $k$ tends to infinity.

Let $p>p_c(\bbZ^3)$. Since the infinite cluster is unique almost
surely, and since there exits an infinite cluster in ${\rm Slab}_k$
for any $k$ sufficiently large (simply choose $k$ so that $p_c(k)<p$), we obtain that
$${\bf P}_p\big[0\stackrel{\bbZ^3}\longleftrightarrow \infty\big]={\bf P}_p\Big[\bigcup_{k\ge0}\{0\stackrel{\slab}\longleftrightarrow \infty\}\Big],$$
from which we deduce that 
\begin{align*}{\bf P}_p\big[0\stackrel{\bbZ^3}\longleftrightarrow \infty\big]&=\lim_{k\rightarrow\infty}{\bf P}_p\big[0\stackrel{\slab}\longleftrightarrow \infty\big]\le \lim_{k\rightarrow\infty}f(p-p_c(k))=f(p-p_c(\bbZ^3)).\end{align*}
As $p$ tends to $p_c(\bbZ^3)$, the continuity of $f$ implies that 
$${\bf P}_{p_c(\bbZ^3)}\big[0\stackrel{\bbZ^3}\longleftrightarrow \infty\big]=0.$$
\end{proof}

\paragraph{Acknowledgements}
We would like to thank I. Benjamini for suggesting this problem to us,
as well as M. Damron, G. Kozma, V. Beffara and C. Newman for very
useful discussions. This project was completed during stays at the
ETH-Z{\"u}rich, Universit{\'e} de Gen{\`e}ve, Weizmann Institute and
IMPA-Rio de Janeiro. We wish to thank all the institutions for their
financial support and hospitality. The work was supported by ESF-RGLIS
network. The first author acknowledges support from the ERC grant AG
CONFRA as well as the FNS. The second author was supported by
Brazilian CNPq grants 308787/2011-0 and 476756/2012-0 and FAPERJ grant
E-26/102.878/2012-BBP. The third author was supported by the ANR grant
ANR-10-BLAN-0123.

\nocite{martineauTassion2013locality}
\bibliographystyle{alpha}
\bibliography{references}
\vspace{2cm}
\raggedleft
{\sc 
  Département de Mathématiques
  \smallskip
  
  Université de Genève
  \smallskip

  Genève, Switzerland
  \smallskip}
  
\texttt{hugo.duminil@unige.ch}
  \bigskip
  
  {\sc IMPA
    \smallskip
    
    Rio de Janeiro, Brazil
    \smallskip}
  
  \texttt{vladas@impa.br}
   \bigskip
  
  {\sc
    UMPA, Ens de Lyon
    \smallskip
    
    Lyon, France
    \smallskip}
  
  \texttt{vincent.tassion@ens-lyon.fr}

\end{document}